\documentclass[12pt]{amsart}
\usepackage{amsmath}
\usepackage{amsfonts}
\usepackage{amssymb}

\usepackage{tikz}
\makeatletter
\def\@seccntformat#1{%
  \protect\textup{%
    \protect\@secnumfont
    \expandafter\protect\csname format#1\endcsname 
    \csname the#1\endcsname
    \protect\@secnumpunct
  }%
}

\begin{document}
\title[Twisted Immanant]
{Twisted Immanant and \\ Matrices with Anticommuting Entries}
\author{Minoru ITOH}
\date{}
\address{Department of Mathematics and Computer Science, 
          Faculty of Science,
          Kagoshima University, Kagoshima 890-0065, Japan}
\email{itoh@sci.kagoshima-u.ac.jp }
\begin{abstract}
   This article gives    
   a new matrix function named ``twisted immanant,''
   which can be regarded as an analogue of the immanant.
   This is defined for each self-conjugate partition
   through a ``twisted'' analogue of the irreducible character of the symmetric group.
   This twisted immanant has some interesting properties.
   For example, it satisfies Cauchy--Binet type formulas.
   Moreover it is closely related to the following results
   for matrices whose entries anticommute with each other:
   (i) the description of the invariants under the conjugations, and 
   (ii) an analogue of the Cauchy identities for symmetric polynomials.
\end{abstract}
\thanks{This research was partially supported by JSPS Grant-in-Aid for Young Scientists (B) 20740020 and JSPS Grant-in-Aid for Young Scientists (B) 24740021.
}
\maketitle
\theoremstyle{plain}
   \newtheorem{theorem}{Theorem}[section]
   \newtheorem{proposition}[theorem]{Proposition}
   \newtheorem{lemma}[theorem]{Lemma}
   \newtheorem{corollary}[theorem]{Corollary}
\theoremstyle{remark}
   \newtheorem*{remark}{Remark}
\numberwithin{equation}{section}
%
%%%%%%%%%%%%%%%%%%%%%%%%%%%%%%%%%%%%%%%%%%%%%%%%%%%%%%%%%%%%%%%%%%%%%%%%%%%%%%%%%%
%
\section{Introduction}
%
%%%%%%%%%%%%%%%%%%%%%%%%%%%%%%%%%%%%%%%%%%%%%%%%%%%%%%%%%%%%%%%%%%%%%%%%%%%%%%%%%%
%
In this article, we introduce a new matrix function named the \emph{twisted immanant}.
This analogue of the immanant is defined by
\[
   \operatorname{imm}^{*\lambda} A
   = \sum_{\sigma \in \mathfrak{S}_n} \chi^{*\lambda}(\sigma) a_{1\sigma(1)} \cdots a_{n\sigma(n)}
\]
for a self-conjugate partition $\lambda$ of~$n$.
Here the notation is as follows.
We fix an associative $\mathbb{C}$-algebra $\mathcal{A}$,
and let $A = (a_{ij})_{1 \leq i,j \leq n}$ be a matrix
whose entries are elements of~$\mathcal{A}$.
Moreover, $\chi^{*\lambda}$ is a certain complex valued function on the symmetric group $\mathfrak{S}_n$,
which is determined by~$\lambda$ naturally 
(the definition will be given in Section~\ref{subsec:chi*}).
This $\chi^{*\lambda}$ is \emph{twisted} in the following sense,
where $\operatorname{sgn}(\tau)$ is the signature of~$\tau$:
\[
   \chi^{*\lambda}(\tau\sigma\tau^{-1})
   = \operatorname{sgn}(\tau) \chi^{*\lambda}(\sigma).
\]
Compare this definition of the twisted immanant
with that of the ordinary immanant:
\[
   \operatorname{imm}^{\lambda} A
   = \sum_{\sigma \in \mathfrak{S}_n} \chi^{\lambda}(\sigma) a_{1\sigma(1)} \cdots a_{n\sigma(n)}.
\]
Here $\chi^{\lambda}$ is the irreducible character of~$\mathfrak{S}_n$,
which is a class function on~$\mathfrak{S}_n$.

%%%%%%%%%%%%%%%%%%%%%%%%%%%%%%%%%%
\subsection{Cauchy--Binet type formulas}
%%%%%%%%%%%%%%%%%%%%%%%%%%%%%%%%%%
The ordinary immanant and the twisted immanant satisfy 
the following Cauchy--Binet type formulas: 

\begin{theorem}
   \label{thm(introduction):Cauchy-Binet_for_imm*}\slshape
   Consider $A \in \operatorname{Mat}_{L,M}(\mathcal{A})$ and 
   $B \in \operatorname{Mat}_{M,N}(\mathcal{A})$,
   and assume that the entries of~$A$ and $B$ commute with each other.
   Then, for~$I \in [L]^n$ and $K \in [N]^n$, we have
   \begin{align*}
      \operatorname{imm}^{\lambda} (AB)_{IK}
      &= \frac{\chi^{\lambda}(1)}{n!} \sum_{J \in [M]^n}
      \operatorname{imm}^{\lambda} A_{IJ} \,
      \operatorname{imm}^{\lambda} B_{JK} \\
      &= \frac{\chi^{\lambda}(1)}{n!} \sum_{J \in [M]^n}
      \operatorname{imm}^{*\lambda} A_{IJ} \,
      \operatorname{imm}^{*\lambda} B_{JK}, 
      \allowdisplaybreaks\\
      \operatorname{imm}^{*\lambda} (AB)_{IK}
      &= \frac{\chi^{\lambda}(1)}{n!} \sum_{J \in [M]^n}
      \operatorname{imm}^{\lambda} A_{IJ} \,
      \operatorname{imm}^{*\lambda} B_{JK} \\
      &= \frac{\chi^{\lambda}(1)}{n!} \sum_{J \in [M]^n}
      \operatorname{imm}^{*\lambda} A_{IJ} \,
      \operatorname{imm}^{\lambda} B_{JK}.
   \end{align*}
\end{theorem}

Here the notation is as follows.
First, we denote by~$\operatorname{Mat}_{m,n}(\mathcal{A})$ the set of all $m \times n$ matrices
whose entries are in~$\mathcal{A}$.
Next, we put $[k] = \{ 1,\ldots,k \}$.
Finally, we put 
$X_{IJ} = (x_{i_s j_t})_{1 \leq s,t \leq n}$ for an $M \times N$ matrix $X = (x_{ij})$
and 
\[
   I = (i_1,\ldots,i_n) \in [M]^n, \qquad
   J = (j_1,\ldots,j_n) \in [N]^n.
\]

%%%%%%%%%%%%%%%%%%%%%%%%%%%%%%%%%%
\subsection{A sum of twisted immanants}
%%%%%%%%%%%%%%%%%%%%%%%%%%%%%%%%%%
We additionally introduce a matrix function defined as a sum of twisted immanants.
Namely, for~$A \in \operatorname{Mat}_{N,N}(\mathcal{A})$ and 
a self-conjugate partition $\lambda$ of~$n$, we put 
\[
   \operatorname{imm}^{*\lambda}_n A 
   = \frac{1}{n!} \sum_{I \in [N]^n} \operatorname{imm}^{*\lambda} A_{II}.
\]
This is invariant under the conjugation by~$GL_N(\mathbb{C})$.
Namely, for any $g \in GL_N(\mathbb{C})$, we have 
\[
   \operatorname{imm}^{*\lambda}_n gAg^{-1}
   =  \operatorname{imm}^{*\lambda}_n A,
\]
even if the entries of~$A$ do not commute with each other.
When the entries of~$A$ commute with each other,
this invariance is a corollary of Theorem~\ref{thm(introduction):Cauchy-Binet_for_imm*}.

%%%%%%%%%%%%%%%%%%%%%%%%%%%%%%%%%%
\subsection{Relation with the trace}
%%%%%%%%%%%%%%%%%%%%%%%%%%%%%%%%%%
The twisted immanant has many applications to matrices with anticommuting entries.
In the remainder of the introduction, we state these applications.

\begin{remark}
   We say that elements of a set $X$ \emph{anticommute} with each other
   when we have $xy = -yx$ for any $x$, $y \in X$ 
   (in particular, $xx = -xx$ for any $x \in X$).
\end{remark}

First, we have the following relation with the trace:

\begin{theorem}
   \label{thm(introduction):imm*_n_and_tr}\slshape
   Consider $A \in \operatorname{Mat}_{N,N}(\mathcal{A})$ 
   whose entries anticommute with each other,
   and put $\mu = (\mu_1,\ldots,\mu_r) = h(\lambda)$ for~$\lambda \in P_{\mathrm{self\text{-}conj}}(n)$.
   Then we have
   \[
      \operatorname{tr}(A^{\mu_1}) \cdots \operatorname{tr}(A^{\mu_r})
      = i^{-m(\lambda)}\sqrt{\mu_1 \cdots \mu_r} \,
      \operatorname{imm}^{*\lambda}_n A.
   \]
   In particular, considering the case $\lambda = (k+1,1^k)$, we have
   \[
      \operatorname{tr}(A^{2k+1})
      = i^{-k} \sqrt{2k+1} \, \operatorname{imm}^{*(k+1,1^{k})}_{2k+1} A.
   \]
\end{theorem}

Here, the notation is as follows.
Firstly, $i$ is the imaginary unit.
Secondly, we put
\begin{align*}
   P_{\mathrm{self\text{-}conj}}(n)
   &= \big\{ \lambda \in P(n) \,\big|\, \text{$\lambda$ is self-conjugate} \big\}, \\
   P_{\mathrm{strict, odd}}(n) &= \big\{ (\mu_1,\ldots,\mu_r) \in P(n)
   \,\big|\, r \geq 0, \,\,\,
   \mu_1 > \cdots > \mu_r > 0, \,\,\, 
   \text{$\mu_1,\ldots,\mu_r$: odd} \big\},
\end{align*}
where $P(n)$ is the set of all partitions of~$n$.
Thirdly, $h$ is the bijection from~$P_{\mathrm{self\text{-}conj}}(n)$ to $P_{\text{strict,odd}}(n)$
defined by
\[
   h \colon 
   \lambda \mapsto (2 \lambda_1 - 1, 2 \lambda_2 - 3,\ldots, 2 \lambda_r - (2r-1)).
\]
Here $r$ is the rank of $\lambda$ (namely the length of the main diagonal of $\lambda$).
Finally, we put $m(\lambda) = \frac{1}{2}(n-r)$ (this quantity is always an integer).

Theorem~\ref{thm(introduction):imm*_n_and_tr} is closely related 
to the following Cayley--Hamilton type theorem 
for an $N \times N$ matrix $A$ whose entries anticommute with each other
(Theorem~\ref{thm:anticommuting_Cayley-Hamilton}):
\begin{equation}\label{eq:anticommuting_Cayley-Hamilton}
   N A^{2N-1} 
   - \operatorname{tr}(A) A^{2N-2} 
   - \operatorname{tr}(A^3) A^{2N-4}
   - \cdots 
   - \operatorname{tr}(A^{2N-3}) A^2 
   - \operatorname{tr}(A^{2N-1}) A^0
   = 0.
\end{equation}
This was given recently in~\cite{BPS} and \cite{I2}.
In this relation,
the elements 
\[
   \operatorname{tr}(A),\operatorname{tr}(A^3),\ldots,\operatorname{tr}(A^{2N-1})
\] 
play the role of the coefficients of the characteristic polynomial. 
Thus, it is expected to describe these elements 
in terms of a determinant-type function,
and this is actually achieved in Theorem~\ref{thm(introduction):imm*_n_and_tr}.
This is one of the motivations for this paper.

Compare this theorem with the ordinary Cayley--Hamilton theorem 
for an $N \times N$ matrix $A$ with commuting entries
(Section~3.10 of~\cite{J}):
\[
   \sum_{0 \leq k \leq N} (-1)^k\operatorname{det}_k(A) A^{N-k} = 0.
\]
Here we put $\operatorname{det}_n(A) = \frac{1}{n!} \sum_{I \in [N]^n} \operatorname{det} A_{II}$.

%%%%%%%%%%%%%%%%%%%%%%%%%%%%%%%%%%
\subsection{{\boldmath$GL(V)$}-invariants in {\boldmath$\Lambda(V \otimes V^*)$}}
%%%%%%%%%%%%%%%%%%%%%%%%%%%%%%%%%%
We can describe the $GL(V)$-invariants in the exterior algebra $\Lambda(V \otimes V^*)$
using the twisted immanants as follows:

\begin{theorem}
   \label{thm(introduction):basis_of_GL-invariants}\slshape
   Let $V$ be a complex vector space of dimension $N$.
   Then the following forms a linear basis of~$\Lambda(V \otimes V^*)^{GL(V)}$\textnormal{:}
   \[
      \big\{ \operatorname{imm}^{*\lambda}_{|\lambda|} A \,\big|\, 
      \text{$\lambda$ is a self-conjugate partition whose first part is not greater than $N$}
      \big\}.
   \]
   Here we define the matrix $A \in \operatorname{Mat}_{N,N}(\Lambda(V \otimes V^*))$ by
   $A = (a_{ij})_{1 \leq i,j \leq N}$ with
   \[
      a_{ij} = e_i \otimes e^*_j \in V \otimes V^* \subset \Lambda(V \otimes V^*),
   \]
   where $\{ e_1,\ldots,e_N \}$ is a basis of~$V$, 
   and $\{ e^*_1,\ldots,e^*_N \}$ is the dual basis.
   Moreover, $|\lambda|$ is the size of $\lambda$.
\end{theorem}

Note that we can regard the matrix $A$ in this theorem as 
the most generic matrix among the square matrices with anticommuting entries.

\begin{remark}
   We can describe the $O(V)$-invariants in the exterior algebra $\Lambda(\Lambda_2(V))$ in a similar way,
   where $V$ is a finite dimensional complex vector space with a nondegenerate symmetric bilinear form.
   This is easily seen from results in~\cite{I2} and Theorem~\ref{thm(introduction):imm*_n_and_tr}.
\end{remark}

%%%%%%%%%%%%%%%%%%%%%%%%%%%%%%%%%%
\subsection{Cauchy type identity}
%%%%%%%%%%%%%%%%%%%%%%%%%%%%%%%%%%
Finally, we have the following relation 
as an analogue of the Cauchy identities on symmetric polynomials.

\begin{theorem}
   \label{thm(introduction):anticommuting_Cauchy_identities}\slshape
   We have
   \begin{align*}
      \operatorname{det}_n (A \otimes B)
      &= \operatorname{per}_n (A \otimes B) \\
      &= \sum_{\lambda \in P_{\mathrm{self\text{-}conj}}(n)} (-1)^{m(\lambda)}
      \operatorname{imm}^{*\lambda}_n A \,
      \operatorname{imm}^{*\lambda}_n B \\
      &= \sum_{(\mu_1,\ldots,\mu_r) \in P_{\mathrm{strict,odd}}(n)} 
      \frac{1}{\mu_1 \cdots \mu_r} 
      \operatorname{tr}(A^{\mu_1}) \cdots \operatorname{tr}(A^{\mu_r})
      \operatorname{tr}(B^{\mu_1}) \cdots \operatorname{tr}(B^{\mu_r}) 
   \end{align*}
   for~$A \in \operatorname{Mat}_{M,M}(\mathcal{A})$ and $B\in \operatorname{Mat}_{N,N}(\mathcal{A})$
   satisfying the following conditions
   \textnormal{(}hence the entries of the Kronecker product $A \otimes B$ commute with each other\textnormal{)}\textnormal{:}
   \begin{itemize}
      \item[(i)]
      the entries of~$A$ anticommute with each other;
      \item[(ii)]
      the entries of~$B$ anticommute with each other;
      \item[(iii)]
      the entries of~$A$ commute with the entries of~$B$.
   \end{itemize}
   Here we define $\operatorname{det}_n X$ and $\operatorname{per}_n X$ for an $N \times N$ matrix $X$ by
   \[
      \operatorname{det}_n X = \frac{1}{n!} \sum_{I \in [N]^n} \operatorname{det} X_{II}, \qquad
      \operatorname{per}_n X = \frac{1}{n!} \sum_{I \in [N]^n} \operatorname{per} X_{II}.
   \]
\end{theorem}

%%%%%%%%%%%%%%%%%%%%%%%%%%%%%%%%%%
\subsection{Related matters}\label{subsec:related_matters}
%%%%%%%%%%%%%%%%%%%%%%%%%%%%%%%%%%
We note some more related matters.
First the Cayley--Hamilton type theorem (\ref{eq:anticommuting_Cayley-Hamilton})
can be regarded as a refinement of the Amitsur--Levitzki theorem 
(see \cite{AL}, \cite{K1}, \cite{K2}, \cite{R} for the Amitsur--Levitzki theorem).
Indeed we see $A^{2N} = 0$ from (\ref{eq:anticommuting_Cayley-Hamilton}),
and the Amitsur--Levitzki theorem is immediate from this relation
(see \cite{I2}, \cite{P} for the detail).
Moreover we can identify the algebra $\Lambda(V \otimes V^*)^{GL(V)}$ 
with the cohomology ring of the Lie algebra $\mathfrak{gl}(V)$ (Chapter~10 of~\cite{Me}).
We note that Kostant gave a proof of the Amitsur--Levitzki theorem
using the cohomology ring of the Lie algebra $\mathfrak{gl}(V)$ 
and the function $\chi^{*\lambda}$ in \cite{K1}.
In this sense, basic ideas of the present article can be found in this paper due to Kostant.

%%%%%%%%%%%%%%%%%%%%%%%%%%%%%%%%%%%%%%%%%%%%%%%%%%%%%%%%%%%%%%%%%%%%%%%%%%%%%%%%%%
%
\section{Ordinary immanant}
%
%%%%%%%%%%%%%%%%%%%%%%%%%%%%%%%%%%%%%%%%%%%%%%%%%%%%%%%%%%%%%%%%%%%%%%%%%%%%%%%%%%
%
First of all, we recall the ordinary immanant and its properties.
The immanant was first defined by Littlewood and Richardson \cite{LR}.
The main reference of this section is Section~6 of~\cite{I1}.

%%%%%%%%%%%%%%%%%%%%%%%%%%%%%%%%
\subsection{Definition of the immanant}
%%%%%%%%%%%%%%%%%%%%%%%%%%%%%%%%
Let us start with the definition.

Fix an associative $\mathbb{C}$-algebra $\mathcal{A}$,
and denote by~$\operatorname{Mat}_{m,n}(\mathcal{A})$ the set of all $m \times n$ matrices
whose entries are elements of~$\mathcal{A}$.
We consider a matrix $A = (a_{ij})_{1 \leq i,j \leq n} \in \operatorname{Mat}_{n,n}(\mathcal{A})$.

Fix a partition $\lambda$ of~$n$.
We define the immanant associated to~$\lambda$ for~$A$ by
\[
   \operatorname{imm}^{\lambda} A
   = \sum_{\tau \in \mathfrak{S}_n} \chi^{\lambda} (\tau) 
   a_{1\tau(1)} \cdots a_{n\tau(n)}.
\]
Here $\chi^{\lambda}$ is the irreducible character of 
the symmetric group $\mathfrak{S}_n$ determined by~$\lambda$.
We can regard this immanant 
as a natural generalization of the determinant and the permanent,
because 
\[
   \operatorname{imm}^{(1^n)} = \det, \qquad
   \operatorname{imm}^{(n)} = \operatorname{per}.
\]

When the matrix entries commute with each other
(for example when $\mathcal{A}$ is commutative),
we can express the immanant in the following various ways:

\begin{proposition}\label{prop:various_expressions_of_imm}\slshape
   When the entries of~$A$ commute with each other,
   we have
   \begin{align*}
      \operatorname{imm}^{\lambda} A
      &= \sum_{\tau \in \mathfrak{S}_n} \chi^{\lambda} (\tau) 
      a_{1\tau(1)} \cdots a_{n\tau(n)} 
      \allowdisplaybreaks\\
      &= \sum_{\sigma \in \mathfrak{S}_n} \chi^{\lambda} (\sigma^{-1}) 
      a_{\sigma(1)1} \cdots a_{\sigma(n)n} 
      \allowdisplaybreaks\\
      &= \frac{\chi^{\lambda}(1)}{n!}
      \sum_{\sigma, \tau \in \mathfrak{S}_n} 
      \chi^{\lambda}(\sigma^{-1}) \chi^{\lambda}(\tau) 
      a_{\sigma(1)\tau(1)} \cdots a_{\sigma(n)\tau(n)} 
      \allowdisplaybreaks\\
      & = \frac{1}{n!}
      \sum_{\sigma, \tau \in \mathfrak{S}_n} 
      \chi^{\lambda} (\tau \sigma^{-1}) 
      a_{\sigma(1)\tau(1)} \cdots a_{\sigma(n)\tau(n)}.
   \end{align*}
\end{proposition}

This proposition is easy from the following relations 
for the irreducible characters (see Section~6 of~\cite{I1} for the detail):
\begin{equation}
\label{eq:relations_for_chi}
   \frac{1}{|\mathfrak{S}_n|} \sum_{\tau \in \mathfrak{S}_n} 
   \chi^{\lambda}(\sigma\tau^{-1}) 
   \chi^{\mu}(\tau)
   =\delta_{\lambda\mu}
   \frac{\chi^{\lambda}(\sigma)}{\chi^{\lambda}(1)}, \qquad
   \chi^{\lambda}(\tau\sigma\tau^{-1})
   = \chi^{\lambda}(\sigma), \qquad
   \chi^{\lambda}(\sigma^{-1}) = \chi^{\lambda}(\sigma).
\end{equation}
Here the first two relations hold in general for any finite group,
and the third holds 
because the irreducible representations of~$\mathfrak{S}_n$ are all real-valued.

The four expressions in Proposition~\ref{prop:various_expressions_of_imm} do not coincide in general,
unless the entries commute with each other.
To distinguish the first and second expressions, we often write them as 
\begin{align*}
   \operatorname{row-imm}^{\lambda} A 
   & = \sum_{\tau \in \mathfrak{S}_n} \chi^{\lambda} (\tau) a_{1\tau(1)} \cdots a_{n\tau(n)},
   \allowdisplaybreaks\\
   \operatorname{column-imm}^{\lambda} A
   & = \sum_{\sigma \in \mathfrak{S}_n} \chi^{\lambda} (\sigma^{-1}) a_{\sigma(1)1} \cdots a_{\sigma(n)n}. 
\end{align*}

%%%%%%%%%%%%%%%%%%%%%%%%%%%%%%%%
\subsection{Cauchy--Binet type formula}
%%%%%%%%%%%%%%%%%%%%%%%%%%%%%%%%
The immanant satisfies the following analogue of the Cauchy--Binet formula
(the case of~$\lambda = (1^n)$ is equal to the ordinary Cauchy--Binet formula for the determinant):

\begin{proposition}\label{prop:Cauchy-Binet_for_imm}\slshape
   Consider $A \in \operatorname{Mat}_{L,M}(\mathcal{A})$ 
   and $B \in \operatorname{Mat}_{M,N}(\mathcal{A})$,
   and assume that the entries of~$A$ and $B$ commute with each other.
   Then, for~$I \in [L]^n$ and $K \in [N]^n$,
   we have 
   \[
      \operatorname{imm}^{\lambda}(AB)_{IK} 
	  = \frac{\chi^{\lambda}(1)}{n!} \sum_{J \in [M]^n}
      \operatorname{imm}^{\lambda}A_{IJ} \,
      \operatorname{imm}^{\lambda}B_{JK}.
   \]
\end{proposition}

\begin{proof}
   We denote the $(i,j)$th entries of~$A$ and $B$ by~$a_{ij}$ and $b_{ij}$, respectively.
   The following calculation using Proposition~\ref{prop:various_expressions_of_imm}
   leads us to the assertion:
   \begin{align*}
      & \operatorname{imm}^{\lambda}(AB)_{IK} \\
      & \qquad
      = \frac{\chi^{\lambda}(1)}{n!} \sum_{\sigma,\tau \in \mathfrak{S}_n} 
      \chi^{\lambda}(\sigma^{-1}) \chi^{\lambda}(\tau)
      (AB)_{i_{\sigma(1)} k_{\tau(1)}} \cdots 
      (AB)_{i_{\sigma(n)} k_{\tau(n)}} 
      \allowdisplaybreaks\\
      & \qquad
      = \frac{\chi^{\lambda}(1)}{n!} \sum_{\sigma,\tau \in \mathfrak{S}_n} 
      \sum_{J \in [M]^n}
      \chi^{\lambda}(\sigma^{-1}) \chi^{\lambda}(\tau)
      a_{i_{\sigma(1)} j_1} b_{j_1 k_{\tau(1)}} 
      \cdots 
      b_{i_{\sigma(n)} j_n} b_{j_n k_{\tau(n)}} 
      \allowdisplaybreaks\\
      & \qquad
      = \frac{\chi^{\lambda}(1)}{n!}
      \sum_{J \in [M]^n} 
      \sum_{\sigma \in \mathfrak{S}_n} 
      \chi^{\lambda}(\sigma^{-1}) 
      a_{i_{\sigma(1)} j_1} \cdots a_{i_{\sigma(n)} j_n}
      \sum_{\tau \in \mathfrak{S}_n} 
      \chi^{\lambda}(\tau)
      b_{j_1 k_{\tau(1)}} \cdots b_{j_n k_{\tau(n)}} 
      \allowdisplaybreaks\\
      & \qquad
      = \frac{\chi^{\lambda}(1)}{n!} 
      \sum_{J \in [M]^n}
      \operatorname{column-imm}^{\lambda}A_{IJ} \,
      \operatorname{row-imm}^{\lambda}B_{JK}. 
   \end{align*}
   Here we denote $(j_1,\ldots,j_n)$ simply by~$J$
   (from now on, we will often use this notation). 
\end{proof}

%%%%%%%%%%%%%%%%%%%%%%%%%%%%%%%%
\subsection{Invariance under permutations}
%%%%%%%%%%%%%%%%%%%%%%%%%%%%%%%%
The immanant has some invariance properties under the permutations of rows and columns. 
Let us put $A^{\sigma} = (a_{\sigma(i)\sigma(j)})_{1 \leq i,j \leq n}$ 
for~$A = (a_{ij})_{1 \leq i,j \leq n}$ and $\sigma \in \mathfrak{S}_n$.
Then, we have 
\[
   \operatorname{imm}^{\lambda} A^{\sigma}
   = \operatorname{imm}^{\lambda} A,
\]
when the entries of~$A$ commute with each other.
Moreover, we have
\begin{align}
\label{eq:imm_for_matrices_with_anticommuting_entries}
   \operatorname{row-imm}^{\lambda}A^{\sigma}
   &= \operatorname{sgn}(\sigma)\operatorname{row-imm}^{\lambda} A, \\
\notag
   \operatorname{column-imm}^{\lambda}A^{\sigma}
   &= \operatorname{sgn}(\sigma)\operatorname{column-imm}^{\lambda} A,
\end{align}
when the entries of~$A$ anticommute with each other.
These relations are immediate from the second relation in~(\ref{eq:relations_for_chi}).

%%%%%%%%%%%%%%%%%%%%%%%%%%%%%%%%
\subsection{A sum of immanants}
%%%%%%%%%%%%%%%%%%%%%%%%%%%%%%%%
We introduce a function defined as a sum of immanants.
We put
\[
   \operatorname{imm}^{\lambda}_n A
   = \frac{1}{n!} \sum_{I \in [N]^n} \operatorname{row-imm}^{\lambda} A_{II}
   = \frac{1}{n!} \sum_{I \in [N]^n} \operatorname{column-imm}^{\lambda} A_{II}.
\]
Here the second equality is seen by a simple calculation.
This function is invariant under the conjugation by~$GL_N(\mathbb{C})$:

\begin{proposition}\label{prop:invariance_of_imm_n}\slshape
   For any $g \in GL_N(\mathbb{C})$, we have
   $\operatorname{imm}^{\lambda}_n gAg^{-1}
   = \operatorname{imm}^{\lambda}_n A$.   
\end{proposition}

This holds even if the entries of~$A$ do not commute with each other.
This proposition is immediate from Proposition~\ref{prop:Cauchy-Binet_for_imm}
when the entries commute with each other.

We also note the following relation with the eigenvalues:

\begin{proposition}\label{prop:imm_n_and_Schur}\slshape
   For~$A \in \operatorname{Mat}_{N,N}(\mathbb{C})$, we have
   \[
      \operatorname{imm}^{\lambda}_n A = s_{\lambda}(x_1,\ldots,x_N).
   \]
   Here $s_{\lambda}$ is the Schur polynomial associated to~$\lambda$,
   and $x_1,\ldots,x_N$ are the eigenvalues of~$A$.
\end{proposition}

See Section~6 of~\cite{I1} for the proofs of Propositions~\ref{prop:invariance_of_imm_n} and 
\ref{prop:imm_n_and_Schur}.

Finally we see the following proposition from~(\ref{eq:imm_for_matrices_with_anticommuting_entries}):

\begin{proposition}\label{prop:imm_n_for_matrices_with_anticommuting_entries}\slshape
   We have $\operatorname{imm}^{\lambda}_n A = 0$ for any $n > 1$,
   when $A$ is a square matrix whose entries anticommute with each other. 
\end{proposition}

%%%%%%%%%%%%%%%%%%%%%%%%%%%%%%%%%%%%%%%%%%%%%%%%%%%%%%%%%%%%%%%%%%%%%%%%%%%%%%%%%%
%
\section{Twisted analogues of the irreducible characters}
%
%%%%%%%%%%%%%%%%%%%%%%%%%%%%%%%%%%%%%%%%%%%%%%%%%%%%%%%%%%%%%%%%%%%%%%%%%%%%%%%%%%
%
To define the \textit{twisted immanant},
we introduce a function $\chi^{*\lambda}$ on~$\mathfrak{S}_n$
satisfying the relation 
$\chi^{*\lambda}(\tau\sigma\tau^{-1}) = \operatorname{sgn}(\tau) \chi^{*\lambda}(\sigma)$.
This $\chi^{*\lambda}$ was first given by Frobenius \cite{F} 
through the representation theory of the alternating group.

See Sections~1.2 and 2.5 of~\cite{JK} for the details in this section.

%%%%%%%%%%%%%%%%%%%%%%%%%%%%%%%%%%%%
\subsection{Notation for partitions}
%%%%%%%%%%%%%%%%%%%%%%%%%%%%%%%%%%%%
Before the main subject, we fix some notation for partitions.

For a partition $\lambda$,
we denote the $i$th part of~$\lambda$ by~$\lambda_i$.
Thus we have $\lambda = (\lambda_1,\ldots,\lambda_l)$ for a partition $\lambda$ of length $l$.

Let $P(n)$ be the set of all partitions of~$n$.
Moreover we put
\begin{align*}
   P_{\mathrm{self\text{-}conj}}(n) &= \big\{ \lambda \in P(n) \,\big|\, \lambda = \lambda' \big\}, \\
   P_{\mathrm{strict, odd}}(n) &= \big\{ (\mu_1,\ldots,\mu_r) \in P(n)
   \,\big|\, r\geq 0, \,\,\,
   \mu_1 > \cdots > \mu_r > 0, \,\, \,
   \text{$\mu_1,\ldots,\mu_r$: odd} \big\},
\end{align*}
where $\lambda'$ means the conjugate of the partition $\lambda$.
Namely $P_{\mathrm{self\text{-}conj}}(n)$ is the set of all self-conjugate partitions of~$n$,
and $P_{\mathrm{strict,odd}}(n)$ is
the set of all strict partitions of~$n$ whose nonzero parts are all odd.
Moreover, we put
\[
   P = \bigsqcup_{n \geq 0} P(n), \qquad
   P_{\mathrm{self\text{-}conj}} = \bigsqcup_{n \geq 0} P_{\mathrm{self\text{-}conj}}(n), \qquad
   P_{\mathrm{strict,odd}} = \bigsqcup_{n \geq 0} P_{\mathrm{strict,odd}}(n).
\]

We have a natural bijection $h$ from~$P_{\mathrm{self\text{-}conj}}(n)$ to~$P_{\mathrm{strict, odd}}(n)$.
Namely, for $\lambda \in P_{\mathrm{self\text{-}conj}}(n)$,
we denote by~$h(\lambda)$ the partition whose $k$th part is equal to the length of the $k$th 
diagonal hook of~$\lambda$.
For example, we have $h \colon (4,4,3,2) \mapsto (7,5,1)$
as is clear from the following correspondence between diagrams:
\smallskip
\[
   \begin{tikzpicture}[x=5mm, y=5mm, baseline = -6mm]
      \fill (0,0) circle (1mm)
            (1,0) circle (1mm)
            (2,0) circle (1mm)
            (3,0) circle (1mm)
            (0,-1) circle (1mm)
            (1,-1) circle (1mm)
            (2,-1) circle (1mm)
            (3,-1) circle (1mm)         
            (0,-2) circle (1mm)
            (1,-2) circle (1mm)
            (2,-2) circle (1mm)
            (0,-3) circle (1mm)
            (1,-3) circle (1mm);
      \draw[-] (3,0) -- (0,0) -- (0,-3);
      \draw[-] (3,-1) -- (1,-1) -- (1,-3);
   \end{tikzpicture}
   \quad \mapsto \quad
   \begin{tikzpicture}[x=5mm, y=5mm, baseline = -6mm]
      \fill (0,0) circle (1mm)
            (1,0) circle (1mm)
            (2,0) circle (1mm)
            (3,0) circle (1mm)
            (4,0) circle (1mm)
            (5,0) circle (1mm)
            (6,0) circle (1mm)
            (0,-1) circle (1mm)
            (1,-1) circle (1mm)
            (2,-1) circle (1mm)
            (3,-1) circle (1mm)         
            (4,-1) circle (1mm)
            (0,-2) circle (1mm);
      \draw[-] (0,0) -- (6,0);
      \draw[-] (0,-1) -- (4,-1);
   \end{tikzpicture}
\smallskip
\]
In other words,
we define $h$ by
\[
   h \colon 
   \lambda \mapsto (2 \lambda_1 - 1, 2 \lambda_2 - 3,\ldots, 2 \lambda_r - (2r-1)).
\]
Here $r$ means the rank of~$\lambda$ 
(namely the length of the main diagonal of~$\lambda$).
It is easily seen that this $h$ is a bijection
from~$P_{\mathrm{self\text{-}conj}}(n)$ to~$P_{\mathrm{strict, odd}}(n)$. 

Finally we put $m(\lambda) = \frac{1}{2}(n-r)$ for a self-conjugate partition 
$\lambda$ of~$n$ of rank $r$.
Note that this quantity is always an integer.

%%%%%%%%%%%%%%%%%%%%%%%%%%%%%%%%%%%%
\subsection{Conjugacy classes of the alternating group}
%%%%%%%%%%%%%%%%%%%%%%%%%%%%%%%%%%%%
We recall the conjugacy classes of the alternating group $\mathfrak{A}_n$.

Let $C_{\mu}$ be the conjugacy class of the symmetric group $\mathfrak{S}_n$ determined by~$\mu \in P(n)$:
\[
   C_{\mu} = \big\{ \sigma \in \mathfrak{S}_n \,\big|\, 
   \text{The cycle type of~$\sigma$ is $\mu$} \big\}.
\]
It is well known that $|C_{\mu}| = n!/z_{\mu}$,
where $z_{\mu} = \prod_j j^{m_j} m_j!$ for~$\mu = (1^{m_1}, 2^{m_2}, \ldots)$.

When $\mu \not\in P_{\mathrm{strict,odd}}(n)$ and $C_{\mu} \subset \mathfrak{A}_n$,
this $C_{\mu}$ is also a conjugacy class of~$\mathfrak{A}_n$.

However, 
when $\mu = (\mu_1,\ldots,\mu_r) \in P_{\mathrm{strict,odd}}(n)$,
then $C_{\mu}$ is a disjoint union of two conjugacy classes of~$\mathfrak{A}_n$.
Let us give more details.
We fix $\sigma \in C_{\mu}$,
and consider the cycle decomposition
\[
   \sigma
   = 
   (i_{11} \,\, i_{12} \,\, \ldots \,\, i_{1\mu_1})
   (i_{21} \,\, i_{22} \,\,\ldots \,\, i_{2\mu_2})
   \cdots
   (i_{r1} \,\, i_{r2} \,\,\ldots \,\, i_{r\mu_r}).
\]
We put $f(\sigma) = 1$ (resp. $f(\sigma) = -1$)
if and only if the inversion number of the following sequence is even (resp. odd):
\[
   (i_{11} \,\, i_{12} \,\, \ldots \,\, i_{1\mu_1} \,\,
   i_{21} \,\, i_{22} \,\,\ldots \,\, i_{2\mu_2} \,\,
   \ldots \,\,
   i_{r1} \,\, i_{r2} \,\,\ldots \,\, i_{r\mu_r}).
\]
We see the well-definedness of~$f(\sigma)$ 
from the fact $\mu \in P_{\mathrm{strict,odd}}(n)$.
Using this, we put
\[
   C^{\pm}_{\mu} = \big\{ \sigma \in C_{\mu} \,\big|\, f(\sigma) = \pm 1 \big\}.
\]
For these $C^{+}_{\mu}$ and $C^{-}_{\mu}$, we have the following:

\begin{proposition}\slshape
   For~$\mu \in P_{\mathrm{strict,odd}}(n)$,
   $C^{+}_{\mu}$ and $C^{-}_{\mu}$ are conjugacy classes of~$\mathfrak{A}_n$ 
   and we have $C_{\mu} = C^{+}_{\mu} \sqcup C^{-}_{\mu}$.
\end{proposition}

%%%%%%%%%%%%%%%%%%%%%%%%%%%%%%%%%%%%
\subsection{Irreducible representations and characters of the alternating group}
%%%%%%%%%%%%%%%%%%%%%%%%%%%%%%%%%%%%
We recall the irreducible representations 
and characters of~$\mathfrak{A}_n$.
Let $\pi^{\lambda}$ be the irreducible representation 
of~$\mathfrak{S}_n$ determined by~$\lambda \in P(n)$,
and consider its restriction to~$\mathfrak{A}_n$.

First, we consider the case $\lambda \not\in P_{\mathrm{self\text{-}conj}}(n)$.
In this case, the restriction $\pi^{\lambda}|_{\mathfrak{A}_n}$ is also irreducible.
We denote by~$\psi^{\lambda}$ the character of~$\pi^{\lambda}|_{\mathfrak{A}_n}$
(namely, $\psi^{\lambda} = \chi^{\lambda}|_{\mathfrak{A}_n}$).
Then the relation $\psi^{\lambda} = \psi^{\lambda'}$ holds,
because $\chi^{\lambda'}(\sigma) = \operatorname{sgn}(\sigma)\chi^{\lambda}(\sigma)$.

Next, we consider the case $\lambda \in P_{\mathrm{self\text{-}conj}}(n)$.
In this case, the restriction $\pi^{\lambda}|_{\mathfrak{A}_n}$ is 
decomposed to two irreducible representations of~$\mathfrak{A}_n$.
We denote by~$\psi^{\lambda \pm}$ 
the characters of these two irreducible representations.
We can describe these $\psi^{\lambda\pm}$ as 
\[
   \psi^{\lambda\pm}(\sigma) = 
   \begin{cases}
      \frac{1}{2} ( (-1)^{m(\lambda)} \pm 
      i^{m(\lambda)} \sqrt{\mu_1 \cdots \mu_r} ), 
      & \sigma \in C^+_{\mu}, \\
      \frac{1}{2} ( (-1)^{m(\lambda)} \mp 
      i^{m(\lambda)} \sqrt{\mu_1 \cdots \mu_r} ),
      & \sigma \in C^-_{\mu}, \\
      \frac{1}{2} \chi^{\lambda}(\sigma), & \sigma \not\in C_{\mu}.
   \end{cases}
\]
Here we put $\mu = h(\lambda)$, and $i$ means the imaginary unit.
We also note the following relation on~$\chi^{\lambda}$ and 
$\sigma \in \mathfrak{S}_n$ whose cycle type is in~$P_{\mathrm{strict, odd}}(n)$:
\[
   \chi^{\lambda}(\sigma) = 
   \begin{cases}
      (-1)^{m(\lambda)}, & \sigma \in C_{\mu}, \\
      0, & \text{otherwise}.
   \end{cases}
\]
Of course, we have $\chi^{\lambda}(\sigma) = \psi^{\lambda +}(\sigma) + \psi^{\lambda -}(\sigma)$.

%%%%%%%%%%%%%%%%%%%%%%%%%%%%%%%%%%%%
\subsection{Definition of {\boldmath$\chi^{* \lambda}$}}\label{subsec:chi*}
%%%%%%%%%%%%%%%%%%%%%%%%%%%%%%%%%%%%
Let us introduce the function $\chi^{*\lambda}\colon \mathfrak{S}_n \to \mathbb{C}$.
For~$\lambda \in P_{\mathrm{self\text{-}conj}}(n)$, 
we put
\[
   \chi^{*\lambda}(\sigma) = 
   \begin{cases}
      \psi^{\lambda+}(\sigma) - \psi^{\lambda-}(\sigma), & \sigma \in \mathfrak{A}_n, \\
      0, & \sigma \not\in \mathfrak{A}_n,
   \end{cases}
\]
so that
\begin{equation}
\label{eq:description_of_chi*}
   \chi^{*\lambda}(\sigma) = 
   \begin{cases}
      \pm i^{m(\lambda)}\sqrt{\mu_1 \cdots \mu_r}, & \sigma \in C^{\pm}_{\mu}, \\
      0, & \sigma \not\in C^{\pm}_{\mu},
   \end{cases}   
\end{equation}
where $\mu = (\mu_1,\ldots,\mu_r) = h(\lambda)$.

For this function $\chi^{*\lambda}$, we have
\begin{equation}
\label{eq:twistedness_of_chi*}
   \chi^{*\lambda}(\sigma^{-1}) 
   = \overline{\chi^{*\lambda}(\sigma)} 
   = \chi^{\lambda}(\sigma) \chi^{*\lambda}(\sigma), \qquad
   \chi^{*\lambda}(\tau\sigma\tau^{-1}) 
   = \operatorname{sgn}(\tau) \chi^{*\lambda}(\sigma)
\end{equation}
and 
\begin{alignat}{2}
\label{eq:orthogonal_relations_for_chi*}
   \frac{1}{|\mathfrak{S}_n|} \sum_{\tau \in \mathfrak{S}_n} 
   \chi^{\lambda}(\sigma\tau^{-1})\chi^{\mu}(\tau)
   &= \delta_{\lambda\mu} \frac{\chi^{\lambda}(\sigma)}{\chi^{\lambda}(1)}, \quad
   &\frac{1}{|\mathfrak{S}_n|} \sum_{\tau \in \mathfrak{S}_n} 
   \chi^{*\lambda}(\sigma\tau^{-1})\chi^{*\mu}(\tau)
   &= \delta_{\lambda\mu} \frac{\chi^{\lambda}(\sigma)}{\chi^{\lambda}(1)}, 
   \allowdisplaybreaks\\
\notag
   \frac{1}{|\mathfrak{S}_n|} \sum_{\tau \in \mathfrak{S}_n} 
   \chi^{*\lambda}(\sigma\tau^{-1})\chi^{\mu}(\tau)
   &= \delta_{\lambda\mu} \frac{\chi^{*\lambda}(\sigma)}{\chi^{\lambda}(1)}, \quad
   &\frac{1}{|\mathfrak{S}_n|} \sum_{\tau \in \mathfrak{S}_n} 
   \chi^{\lambda}(\sigma\tau^{-1})\chi^{*\mu}(\tau)
   &= \delta_{\lambda\mu} \frac{\chi^{*\lambda}(\sigma)}{\chi^{\lambda}(1)}.
\end{alignat}
Here the superscripts of~$\chi^*$ are self-conjugate.
Moreover, $\{ \chi^{*\lambda} \,|\, \lambda \in P_{\mathrm{self\text{-}conj}}(n) \}$ 
forms an orthogonal basis of the following vector space:
\[
   \big\{ \chi \colon \mathfrak{S}_n \to \mathbb{C} \,\big|\, 
   \text{$\chi(\tau\sigma\tau^{-1}) = \operatorname{sgn}(\tau) \chi(\sigma)$ 
   for any $\sigma$, $\tau \in \mathfrak{S}_n$} \big\}.
\]
These properties follow from the fact that $\psi^{\lambda\pm}$ is 
an irreducible character of~$\mathfrak{A}_n$.

%%%%%%%%%%%%%%%%%%%%%%%%%%%%%%%%%%%%%%%%%%%%%%%%%%%%%%%%%%%%%%%%%%%%%%%%%%%%%%%%%%
%
\section{Twisted immanant}
%
%%%%%%%%%%%%%%%%%%%%%%%%%%%%%%%%%%%%%%%%%%%%%%%%%%%%%%%%%%%%%%%%%%%%%%%%%%%%%%%%%%
%
Let us introduce the twisted immanant and see its basic properties.

%%%%%%%%%%%%%%%%%%%%%%%%%%%%%%%%%%%%
\subsection{Definition of the twisted immanant}
%%%%%%%%%%%%%%%%%%%%%%%%%%%%%%%%%%%%
We define the twisted immanant using the function $\chi^{*\lambda}$.
For~$\lambda \in P_{\mathrm{self\text{-}conj}}(n)$ and 
$A = (a_{ij})_{1 \leq i,j \leq n} \in \operatorname{Mat}_{n,n}(\mathcal{A})$, 
we put 
\[
   \operatorname{imm}^{*\lambda} A
   = \sum_{\tau \in \mathfrak{S}_n} \chi^{*\lambda}(\tau)
   a_{1 \tau(1)} \cdots a_{n \tau(n)}.
\]
When the entries of~$A$ commute with each other,
we can express this in the following various way:
\begin{align*}
   \operatorname{imm}^{*\lambda} A
   &= \sum_{\tau \in \mathfrak{S}_n} \chi^{*\lambda}(\tau)
   a_{1 \tau(1)} \cdots a_{n \tau(n)}
   \allowdisplaybreaks\\
   &= \sum_{\sigma \in \mathfrak{S}_n} \chi^{*\lambda}(\sigma^{-1})
   a_{\sigma(1) 1} \cdots a_{\sigma(n) n} 
   \allowdisplaybreaks\\
   &= \frac{1}{n!} \sum_{\sigma,\tau \in \mathfrak{S}_n} \chi^{*\lambda}(\tau\sigma^{-1})
   a_{\sigma(1)\tau(1)} \cdots a_{\sigma(p) \tau(p)} 
   \allowdisplaybreaks\\
   &= \frac{\chi^{\lambda}(1)}{n!} 
   \sum_{\sigma,\tau \in \mathfrak{S}_n} \chi^{*\lambda}(\tau) \chi^{\lambda}(\sigma^{-1})
   a_{\sigma(1)\tau(1)} \cdots a_{\sigma(p) \tau(p)} 
   \allowdisplaybreaks\\
   &= \frac{\chi^{\lambda}(1)}{n!} 
   \sum_{\sigma,\tau \in \mathfrak{S}_n} \chi^{\lambda}(\tau) \chi^{*\lambda}(\sigma^{-1})
   a_{\sigma(1)\tau(1)} \cdots a_{\sigma(p) \tau(p)}.
\end{align*}
These equalities also hold
when the entries of~$A$ anticommute with each other.
This is seen by noting that $\chi^{*\lambda}(\sigma) = 0$ unless $\sigma \in \mathfrak{A}_n$.
However, unless the entries commute or anticommute with each other,
these equalities do not hold in general.
To distinguish the first and second expressions, we often denote them as follows:
\begin{align*}
   \operatorname{row-imm}^{*\lambda} A 
   &= \sum_{\tau \in \mathfrak{S}_n} \chi^{*\lambda}(\tau)
   a_{1 \tau(1)} \cdots a_{n \tau(n)}, \\
   \operatorname{column-imm}^{*\lambda} A 
   &= \sum_{\sigma \in \mathfrak{S}_n} \chi^{*\lambda}(\sigma^{-1})
   a_{\sigma(1) 1} \cdots a_{\sigma(n) n}.
\end{align*}

%%%%%%%%%%%%%%%%%%%%%%%%%%%%%%%%
\subsection{Cauchy--Binet type formulas}
%%%%%%%%%%%%%%%%%%%%%%%%%%%%%%%%
The first remarkable property of the twisted immanant is 
Theorem~\ref{thm(introduction):Cauchy-Binet_for_imm*}.
This can be regarded
as Cauchy--Binet type formulas for the ordinary immanant and the twisted immanant.
We can prove this using (\ref{eq:orthogonal_relations_for_chi*}) 
in a way similar to the proof of Proposition~\ref{prop:Cauchy-Binet_for_imm}.

%%%%%%%%%%%%%%%%%%%%%%%%%%%%%%%%%%%%
\subsection{Relation with the conjugate transpose}
%%%%%%%%%%%%%%%%%%%%%%%%%%%%%%%%%%%%
The twisted immanant satisfies the following relation for~$A \in \operatorname{Mat}_{n,n}(\mathbb{C})$:
\[
   \operatorname{imm}^{*\lambda} A^* = \overline{\operatorname{imm}^{*\lambda} A}.
\]
Here we denote by~$A^*$ the conjugate transpose of~$A$
(namely we put $A^* = (\bar{a}_{ji})_{1 \leq i,j \leq n}$ for~$A = (a_{ij})_{1 \leq i,j \leq n}$).
This relation is immediate from the first relation of~(\ref{eq:twistedness_of_chi*}).

%%%%%%%%%%%%%%%%%%%%%%%%%%%%%%%%%%%%
\subsection{Invariance under permutations}
%%%%%%%%%%%%%%%%%%%%%%%%%%%%%%%%%%%%
The twisted immanant has some invariance properties under the permutations of rows and columns
as the ordinary immanant does. 
Consider an $n \times n$ matrix $A$ and $\sigma \in \mathfrak{S}_n$.
On the one hand,
when the entries of~$A$ commute with each other,
we have 
\begin{equation}
\label{eq:twistedness_of_imm*}
   \operatorname{imm}^{*\lambda} A^{\sigma}
   = \operatorname{sgn}(\sigma) \operatorname{imm}^{*\lambda} A.
\end{equation}
On the other hand, when the entries of~$A$ anticommute with each other, we have 
\[
   \operatorname{imm}^{*\lambda} A^{\sigma} 
   = \operatorname{imm}^{*\lambda} A.
\]
These relations are immediate from the second relation of~(\ref{eq:twistedness_of_chi*}).

%%%%%%%%%%%%%%%%%%%%%%%%%%%%%%%%%%%%
\subsection{A sum of twisted immanants}
%%%%%%%%%%%%%%%%%%%%%%%%%%%%%%%%%%%%
Let us consider the counterpart of the function ``$\operatorname{imm}^{\lambda}_n$.''
For~$\lambda \in P_{\mathrm{self\text{-}conj}}(n)$ and $A \in \operatorname{Mat}_{N,N}(\mathcal{A})$, 
we put 
\[
   \operatorname{imm}^{*\lambda}_n A 
   = \frac{1}{n!} \sum_{I \in [N]^n} \operatorname{row-imm}^{*\lambda} A_{II} 
   = \frac{1}{n!} \sum_{I \in [N]^n} \operatorname{column-imm}^{*\lambda} A_{II}.
\]
Here the second equality is seen by a direct calculation.
This is invariant under the conjugation by~$GL_N(\mathbb{C})$:

\begin{theorem}\slshape
   For any $g \in GL_N(\mathbb{C})$, we have
   $\operatorname{imm}^{*\lambda}_n gAg^{-1}
   = \operatorname{imm}^{*\lambda}_n A$.
\end{theorem}

To prove this, it suffices to prove the following lemma:

\begin{lemma}\slshape
   For any $g \in \operatorname{Mat}_{N,N}(\mathbb{C})$, we have
   $\operatorname{imm}^{*\lambda}_n gA
   = \operatorname{imm}^{*\lambda}_n Ag$.
\end{lemma}

\begin{proof}
   We denote the $(i,j)$th entries of~$A$ and $g$ by~$a_{ij}$ and $g_{ij}$, respectively.
   We see the assertion from the following calculation:
   \begin{align*}
      \operatorname{imm}^{*\lambda}_n gA
      &= \frac{1}{n!} \sum_{I \in [N]^n} \operatorname{column-imm}^{*\lambda} (gA)_{II} 
      \allowdisplaybreaks\\
      &= \frac{1}{n!} \sum_{I \in [N]^n} \sum_{\sigma \in \mathfrak{S}_n}
      \chi^{\lambda}(\sigma^{-1})
      (gA)_{i_{\sigma(1)} i_1} (gA)_{i_{\sigma(2)} i_2}\cdots 
      (gA)_{i_{\sigma(n)} i_n} 
      \allowdisplaybreaks\\
      &= \frac{1}{n!} \sum_{I \in [N]^n} \sum_{J \in [N]^n} \sum_{\sigma \in \mathfrak{S}_n}
      \chi^{\lambda}(\sigma^{-1})
      g_{i_{\sigma(1)} j_1} a_{j_1 i_1} 
      g_{i_{\sigma(2)} j_2} a_{j_2 i_2} \cdots 
      g_{i_{\sigma(n)} j_n} a_{j_n i_n} 
      \allowdisplaybreaks\\
      &= \frac{1}{n!} \sum_{I \in [N]^n} \sum_{J \in [N]^n} \sum_{\sigma \in \mathfrak{S}_n}
      \chi^{\lambda}(\sigma^{-1})
      a_{j_1 i_1} g_{i_1 j_{\sigma^{-1}(1)}}
      a_{j_2 i_2} g_{i_2 j_{\sigma^{-1}(2)}} \cdots 
      a_{j_n i_n} g_{i_n j_{\sigma^{-1}(n)}} 
      \allowdisplaybreaks\\
      &= \frac{1}{n!} \sum_{I \in [N]^n} \sum_{J \in [N]^n} \sum_{\tau \in \mathfrak{S}_n}
      \chi^{\lambda}(\tau)
      a_{j_1 i_1} g_{i_1 j_{\tau(1)}}
      a_{j_2 i_2} g_{i_2 j_{\tau(2)}} \cdots 
      a_{j_n i_n} g_{i_n j_{\tau(n)}} 
      \allowdisplaybreaks\\
      &= \frac{1}{n!} \sum_{J \in [N]^n} \sum_{\tau \in \mathfrak{S}_n}
      \chi^{\lambda}(\tau)
      (Ag)_{j_1 j_{\tau(1)}} 
      (Ag)_{j_2 j_{\tau(2)}} \cdots 
      (Ag)_{j_n j_{\tau(n)}} 
      \allowdisplaybreaks\\
      &= \frac{1}{n!} \sum_{J \in [N]^n} \operatorname{row-imm}^{*\lambda} (Ag)_{JJ} 
      \allowdisplaybreaks\\
      &= \operatorname{imm}^{*\lambda}_n Ag.
   \end{align*}
   Here, in the fourth equality, we changed the positions of the entries of~$g$.
   Moreover, in the fifth equality, we replaced $\sigma^{-1}$ by~$\tau$.
\end{proof}

However, the function $\operatorname{imm}^{*\lambda}_n$ 
is almost trivial for matrices with commuting entries.
Indeed we see the following from~(\ref{eq:twistedness_of_imm*}).
This can be regarded as the counterpart of 
Proposition~\ref{prop:imm_n_for_matrices_with_anticommuting_entries}.

\begin{proposition}\slshape
   We have $\operatorname{imm}^{*\lambda}_n A = 0$ for~$n>1$,
   when the entries of~$A$ commute with each other.
\end{proposition}

%%%%%%%%%%%%%%%%%%%%%%%%%%%%%%%%%%%%%%%%%%%%%%%%%%%%%%%%%%%%%%%%%%%%%%%%%%%%%%%%%%
%
\section{Twisted immanant and matrices with anticommuting entries}
\label{sec:matrices_with_anticommuting_entries}
%
%%%%%%%%%%%%%%%%%%%%%%%%%%%%%%%%%%%%%%%%%%%%%%%%%%%%%%%%%%%%%%%%%%%%%%%%%%%%%%%%%%
%
In the remainder of this article, we deal with the twisted immanant
for matrices with anticommuting entries.

In this section, we look at Theorem~\ref{thm(introduction):imm*_n_and_tr},
a relation between $\operatorname{imm}^{*\lambda}_n$ and the traces of powers.
We already saw that the function $\operatorname{imm}^{*\lambda}_n$ 
is almost trivial for matrices with commuting entries,
but this function has interesting properties for matrices with anticommuting entries.

We prove Theorem~\ref{thm(introduction):imm*_n_and_tr}
using the following easy lemma (see \cite{R} for the proof):

\begin{lemma}\label{lem:trace_of_A^2k}\slshape
   When the entries of an $N \times N$ matrix $A$ anticommute with each other, we have 
   $\operatorname{tr}(A^2) = \operatorname{tr}(A^4) = \operatorname{tr}(A^6) = \cdots = 0$.
\end{lemma}

\begin{proof}[Proof of Theorem~\textsl{\ref{thm(introduction):imm*_n_and_tr}}]
   We denote the $(i,j)$th entry of~$A$ by~$a_{ij}$.
   Then we have
   \begin{align*}
      \operatorname{imm}^{* \lambda}_n A
      &= \frac{1}{n!} \sum_{I \in [N]^n} \operatorname{row-imm}^{* \lambda} A_{II} \\
      &= \frac{1}{n!} \sum_{I \in [N]^n} \sum_{\tau \in \mathfrak{S}_n} 
      \chi^{* \lambda}(\tau) 
      a_{i_1 i_{\tau(1)}} \cdots a_{i_n i_{\tau(n)}} \\
      &= \frac{1}{n!} 
      \sum_{\tau \in \mathfrak{S}_n}  
      \chi^{* \lambda}(\tau) 
      \sum_{I \in [N]^n}
      a_{i_1 i_{\tau(1)}} \cdots a_{i_n i_{\tau(n)}}.
   \end{align*}
   Denoting the cycle type of~$\tau$ by~$\mu = (\mu_1,\ldots,\mu_r)$, we have
   \[
      \sum_{I \in [N]^n}
      a_{i_1 i_{\tau(1)}} \cdots a_{i_n i_{\tau(n)}}
      = f(\tau)
      \operatorname{tr}(A^{\mu_1}) \cdots \operatorname{tr}(A^{\mu_r}).
   \]
   This vanishes unless $\mu \in P_{\mathrm{strict,odd}}(n)$ 
   as seen from Lemma~\ref{lem:trace_of_A^2k}.
   Since $|C_{\mu}| = n! / z_{\mu} = n! / \mu_1 \cdots \mu_r$ 
   for~$\mu \in P_{\mathrm{strict,odd}}(n)$, 
   we see the assertion from~(\ref{eq:description_of_chi*}).
\end{proof}

%%%%%%%%%%%%%%%%%%%%%%%%%%%%%%%%%%%%%%%%%%%%%%%%%%%%%%%%%%%%%%%%%%%%%%%%%%%%%%%%%%
%
\section{$GL(V)$-invariants in $\Lambda(V \otimes V^*)$}
\label{sec:GL-invariants}
%
%%%%%%%%%%%%%%%%%%%%%%%%%%%%%%%%%%%%%%%%%%%%%%%%%%%%%%%%%%%%%%%%%%%%%%%%%%%%%%%%%%
%
In this section,
we develop Theorem~\ref{thm(introduction):imm*_n_and_tr} to
a description of $\Lambda(V \otimes V^*)^{GL(V)}$,
the set of all $GL(V)$-invariants in the exterior algebra $\Lambda(V \otimes V^*)$.
Namely we prove Theorem~\ref{thm(introduction):basis_of_GL-invariants},
a description of these invariants in terms of the twisted immanant.

Let $V$ be a complex vector space of dimension $N$.
We put 
\[
   a_{ij} = e_i \otimes e^*_j \in V \otimes V^* \subset \Lambda(V \otimes V^*),
\]
where $\{ e_1,\ldots,e_N \}$ are a basis of~$V$, 
and $\{ e^*_1,\ldots,e^*_N \}$ are the dual basis.
Let us consider the matrix 
$A = (a_{ij})_{1 \leq i,j \leq N} \in \operatorname{Mat}_{N,N}(\Lambda(V \otimes V^*))$.
We can regard this matrix as the most generic matrix among the square matrices with anticommuting entries. 
Using this matrix, we can describe $\Lambda(V \otimes V^*)^{GL(V)}$ as follows
(this is a consequence of the first fundamental theorem of invariant theory for vector invariants;
see \cite{I2} for the proof):

\begin{theorem}\label{thm:generators_of_GL-invariants}\slshape
   The following elements generate $\Lambda(V \otimes V^*)^{GL(V)}$\textnormal{:}
   \[
      \operatorname{tr}(A), \,
      \operatorname{tr}(A^3), \,
      \ldots, \,
      \operatorname{tr}(A^{2N-3}), \,
      \operatorname{tr}(A^{2N-1}).
   \]
   Moreover 
   these elements anticommute with each other,
   and have no relations besides this anticommutativity.
   Namely the following forms a linear basis of~$\Lambda(V \otimes V^*)^{GL(V)}$\textnormal{:}
   \[
      \big\{ \operatorname{tr}(A^{\mu_1}) \cdots \operatorname{tr}(A^{\mu_r}) \,\big|\,
      (\mu_1,\ldots,\mu_r) \in P_{\mathrm{strict,odd}},\,\,
      \mu_1 < 2N \big\}.
   \]
\end{theorem}

Combining this with Theorem~\ref{thm(introduction):imm*_n_and_tr},
we have Theorem~\ref{thm(introduction):basis_of_GL-invariants}.

The results in Sections~\ref{sec:matrices_with_anticommuting_entries} and \ref{sec:GL-invariants}
are closely related to the following Cayley--Hamilton type theorem
(\cite{BPS}, \cite{P}, \cite{I2}; see also \cite{DPP}, \cite{D}):

\begin{theorem}\label{thm:anticommuting_Cayley-Hamilton}\slshape
   We have
   \[
      N A^{2N-1} 
      - \operatorname{tr}(A) A^{2N-2} 
      - \operatorname{tr}(A^3) A^{2N-4}
      - \cdots 
      - \operatorname{tr}(A^{2N-3}) A^2 
      - \operatorname{tr}(A^{2N-1}) A^0
      = 0.
   \]
\end{theorem}

In this equality, the elements
$\operatorname{tr}(A),\operatorname{tr}(A^3),\ldots,\operatorname{tr}(A^{2N-1})$ 
play the role of coefficients of the characteristic polynomial. 
Thus, it is expected to describe $\operatorname{tr}(A^{2k+1})$
in terms of a determinant-type function,
and this is actually achieved in Theorem~\ref{thm(introduction):imm*_n_and_tr}.
This is one of the motivations for this paper.

As written in Section~\ref{subsec:related_matters},
Theorem~\ref{thm:anticommuting_Cayley-Hamilton} can be regarded as a refinement of
the Amitsur--Levitzki theorem.
We note that Kostant proved the Amitsur--Levitzki theorem
using the cohomology ring of the Lie algebra $\mathfrak{gl}(V)$ 
(which is isomorphic to~$\Lambda(V \otimes V^*)^{GL(V)}$)
and the function $\chi^{*\lambda}$ in~\cite{K1}.
In this sense, basic ideas of the present article can be found in this paper due to Kostant.

%%%%%%%%%%%%%%%%%%%%%%%%%%%%%%%%%%%%%%%%%%%%%%%%%%%%%%%%%%%%%%%%%%%%%%%%%%%%%%%%%%
%
%%\section{Cauchy type relation}
\section{Cauchy type identity}
%
%%%%%%%%%%%%%%%%%%%%%%%%%%%%%%%%%%%%%%%%%%%%%%%%%%%%%%%%%%%%%%%%%%%%%%%%%%%%%%%%%%
%
Finally, we discuss Theorem~\ref{thm(introduction):anticommuting_Cauchy_identities},
namely an Cauchy type identity for the twisted immanant.
Let us start with the ordinary Cauchy identities for symmetric polynomials (Section~I.4 of~\cite{Ma}):

\begin{proposition}\label{prop:Cauchy_identities(polynomial)}\slshape
   We have
   \begin{align*}
      \prod_{1 \leq i \leq M} \prod_{1 \leq j \leq N} \frac{1}{1 - x_i y_j} 
      &= \sum_{\lambda \in P} s_{\lambda}(x_1,\ldots,x_M) s_{\lambda}(y_1,\ldots,y_N) \\
      &= \sum_{\mu \in P} \frac{1}{z_{\mu}}p_{\mu}(x_1,\ldots,x_M) p_{\mu}(y_1,\ldots,y_N), 
      \allowdisplaybreaks\\
      \prod_{1 \leq i \leq M} \prod_{1 \leq j \leq N} (1 + x_i y_j) 
      &= \sum_{\lambda \in P} s_{\lambda}(x_1,\ldots,x_M) s_{\lambda'}(y_1,\ldots,y_N) \\
      &= \sum_{\mu \in P} (-1)^{n-r} \frac{1}{z_{\mu}} 
      p_{\mu}(x_1,\ldots,x_M) p_{\mu}(y_1,\ldots,y_N).
   \end{align*}
   Here $r$ is the length of~$\mu$.
\end{proposition}

These can be rewritten in terms of matrices as follows:

\begin{proposition}\label{prop:Cauchy_identities(matrix)}\slshape
   Consider $A \in \operatorname{Mat}_{M,M}(\mathcal{A})$ and 
   $B \in \operatorname{Mat}_{N,N}(\mathcal{A})$.
   When the entries of~$A$ and $B$ commute with each other, we have
   \begin{align*}
      \operatorname{per}_n(A \otimes B)
      &= \sum_{\lambda \in P(n)} \operatorname{imm}^{\lambda}_n A \,
      \operatorname{imm}^{\lambda}_n B \\
      &= \sum_{\mu = (\mu_1,\ldots,\mu_r) \in P(n)} \frac{1}{z_{\mu}} 
      \operatorname{tr}(A^{\mu_1}) \cdots \operatorname{tr}(A^{\mu_r})
      \operatorname{tr}(B^{\mu_1}) \cdots \operatorname{tr}(B^{\mu_r}), 
      \allowdisplaybreaks\\
      \operatorname{det}_n(A \otimes B)
      &= \sum_{\lambda \in P(n)} \operatorname{imm}^{\lambda}_n A \,
      \operatorname{imm}^{\lambda'}_n B \\
      &= \sum_{\mu = (\mu_1,\ldots,\mu_r) \in P(n)} (-1)^{n-r} \frac{1}{z_{\mu}} 
      \operatorname{tr}(A^{\mu_1}) \cdots \operatorname{tr}(A^{\mu_r})
      \operatorname{tr}(B^{\mu_1}) \cdots \operatorname{tr}(B^{\mu_r}).
   \end{align*}
\end{proposition}

Here we put 
$\operatorname{det}_n = \operatorname{imm}^{(1^n)}_n$ and 
$\operatorname{per}_n = \operatorname{imm}^{(n)}_n$.
Moreover $A \otimes B$ is the Kronecker product. 
Namely we put
$A \otimes B = (a_{ij}b_{kl})_{(i,k), (j,l) \in [M] \times [N]}$
for~$A = (a_{ij})_{1 \leq i,j \leq M}$
and $B = (b_{ij})_{1 \leq i,j \leq N}$.

Proposition~\ref{prop:Cauchy_identities(matrix)} follows 
from Proposition~\ref{prop:Cauchy_identities(polynomial)} 
by using Proposition~\ref{prop:imm_n_and_Schur} and the relation
\[
   \operatorname{tr}(A^k) = p_k(x_1,\ldots,x_N).
\]
Here $A$ is an $N \times N$ complex matrix,
and $x_1,\ldots,x_N$ are the eigenvalues of~$A$.

Theorem~\ref{thm(introduction):anticommuting_Cauchy_identities} 
is quite similar to Proposition~\ref{prop:Cauchy_identities(matrix)}.
Thus, we can regard Theorem~\ref{thm(introduction):anticommuting_Cauchy_identities} 
as an anticommuting analogue of the Cauchy identities.

\begin{proof}[Proof of Theorem~\textsl{\ref{thm(introduction):anticommuting_Cauchy_identities}}]
   Let $a_{ij}$ and $b_{ij}$ denote the $(i,j)$th entries of~$A$ and $B$, respectively.
   Then we have
   \begin{align*}
      \operatorname{det}_n (A \otimes B)
      &= \frac{1}{n!}
      \sum_{I \in [M]^n, \, J \in [N]^n} \sum_{\sigma \in \mathfrak{S}_n} \operatorname{sgn}(\sigma)
      a_{i_1i_{\sigma(1)}} b_{j_1 j_{\sigma(1)}} \cdots  a_{i_n i_{\sigma(n)}} b_{j_n j_{\sigma(n)}} \\
      &= \frac{1}{n!}
      \sum_{\sigma \in \mathfrak{S}_n} \operatorname{sgn}(\sigma)
      \sum_{I \in [M]^n} a_{i_1 i_{\sigma(1)}} \cdots a_{i_n i_{\sigma(n)}} 
      \sum_{J \in [N]^n} b_{j_1 j_{\sigma(1)}} \cdots b_{j_n j_{\sigma(n)}}, \\
      \operatorname{per}_n (A \otimes B)
      &= \frac{1}{n!}
      \sum_{I \in [M]^n, \, J \in [N]^n} \sum_{\sigma \in \mathfrak{S}_n} 
      a_{i_1 i_{\sigma(1)}} b_{j_1 j_{\sigma(1)}} \cdots  a_{i_n i_{\sigma(n)}} b_{j_n j_{\sigma(n)}} \\
      &= \frac{1}{n!}
      \sum_{\sigma \in \mathfrak{S}_n}
      \sum_{I \in [M]^n} a_{i_1 i_{\sigma(1)}} \cdots a_{i_n i_{\sigma(n)}} 
      \sum_{J \in [N]^n} b_{j_1 j_{\sigma(1)}} \cdots b_{j_n j_{\sigma(n)}}.
   \end{align*}
   Denoting the cycle type of~$\sigma$ by~$\mu = (\mu_1,\ldots,\mu_r)$, we have
   \begin{align*}
      \sum_{I \in [M]^n}
      a_{i_1i_{\sigma(1)}} \cdots a_{i_n i_{\sigma(n)}}
      &= f(\sigma) \operatorname{tr}(A^{\mu_1}) \cdots \operatorname{tr}(A^{\mu_r}), \\
      \sum_{J \in [N]^n}
      b_{j_1 j_{\sigma(1)}} \cdots b_{j_n j_{\sigma(n)}} 
      &= f(\sigma) \operatorname{tr}(B^{\mu_1}) \cdots \operatorname{tr}(B^{\mu_r}).
   \end{align*}
   These quantities vanish unless $\mu \in P_{\mathrm{strict,odd}}(n)$
   as seen in Lemma~\ref{lem:trace_of_A^2k}.
   Moreover, we have $|C_{\mu}| = n! / z_{\mu} = n! / \mu_1 \cdots \mu_r$
   for~$\mu \in P_{\mathrm{strict,odd}}(n)$.
   Combining these facts, we have 
   \begin{align*}
      \operatorname{det}_n (A \otimes B)
      &= \operatorname{per}_n (A \otimes B) \\
      &=
      \frac{1}{n!}
      \sum_{\mu \in P_{\mathrm{strict,odd}}(n)} \frac{n!}{\mu_1 \cdots \mu_r}
      \operatorname{tr}(A^{\mu_1}) \cdots \operatorname{tr}(A^{\mu_r})
      \operatorname{tr}(B^{\mu_1}) \cdots \operatorname{tr}(B^{\mu_r}).
   \end{align*}
   The remainder of the assertion follows from Theorem~\ref{thm(introduction):basis_of_GL-invariants}.
\end{proof}

\begin{remark}
   A similar result holds,
   even if we replace the condition (iii) 
   in Theorem~\ref{thm(introduction):anticommuting_Cauchy_identities} by the following (iii${}'$):
   \begin{itemize}
      \item[(iii${}'$)]
      the entries of~$A$ anticommute with the entries of~$B$.
   \end{itemize}
   Indeed, under the conditions (i), (ii) and (iii${}'$), we have
   \begin{align*}
      \operatorname{det}_n (A \otimes B)
      &= \operatorname{per}_n (A \otimes B) \\
      &= \varepsilon_n
      \sum_{\lambda \in P_{\mathrm{self\text{-}conj}}(n)} (-1)^{m(\lambda)}
      \operatorname{imm}^{*\lambda}_n A \,
      \operatorname{imm}^{*\lambda}_n B \\
      &= \varepsilon_n 
      \sum_{(\mu_1,\ldots,\mu_r) \in P_{\mathrm{strict,odd}}(n)} \frac{1}{\mu_1 \cdots \mu_r} 
      \operatorname{tr}(A^{\mu_1}) \cdots \operatorname{tr}(A^{\mu_r})
      \operatorname{tr}(B^{\mu_1}) \cdots \operatorname{tr}(B^{\mu_r})
   \end{align*}
   by a similar proof.
   Here $\varepsilon_n$ means
   \[
      \varepsilon_n = 
      \begin{cases}
         1, & n = 0,1 \mod 4, \\
         -1, & n = 2,3 \mod 4.
      \end{cases}
   \]
\end{remark}

%%%%%%%%%%%%%%%%%%%%%%%%%%%%%%%%%%%%%%%%%%%%%%%%%%%%%%%%%%%%%%%%%%%%%%%%%%%%%%%%%%
%
\section*{Acknowledgments}
%
%%%%%%%%%%%%%%%%%%%%%%%%%%%%%%%%%%%%%%%%%%%%%%%%%%%%%%%%%%%%%%%%%%%%%%%%%%%%%%%%%%
%
The author is grateful for the hospitality of the Max-Planck-Institut f\"ur Mathematik in Bonn, 
where the essential part of this work started. 
The author is also grateful to the referee for the valuable comments.

%%%%%%%%%%%%%%%%%%%%%%%%%%%%%%%%%%%%%%%%%%%%%%%%%%%%%%%%%%%%%%%%%%%%%%%%%%
%
% References
%
%%%%%%%%%%%%%%%%%%%%%%%%%%%%%%%%%%%%%%%%%%%%%%%%%%%%%%%%%%%%%%%%%%%%%%%%%%
%

\end{document}